\documentclass[reqno,10pt,centertags, draft]{amsart}
\usepackage{amsmath,amsthm,amscd,amssymb,latexsym,upref,mathrsfs}
\usepackage{color}

\newcommand{\dC}{{\mathbb{C}}}
\newcommand{\dR}{{\mathbb{R}}}

\newcommand{\cL}{{\mathcal L}}

\def\senki{{\lbrack\negthinspace [\bot ]\negthinspace\rbrack}}
\def\senki+{{\lbrack\negthinspace [+] \negthinspace\rbrack}}

\DeclareMathOperator{\dom}{dom}
\DeclareMathOperator{\ran}{ran}

\allowdisplaybreaks \numberwithin{equation}{section}

\newtheorem{theorem}{Theorem}[section]

\newtheorem{lemma}[theorem]{Lemma}

\theoremstyle{remark}

\begin{document}

\title[Compressed Resolvents of Schr\"{o}dinger Operators]{On Compressed Resolvents of Schr\"{o}dinger Operators with Complex Potentials}

\author[J.\ Behrndt]{Jussi Behrndt}
\address{Institut f\"ur Angewandte Mathematik, Technische Universit\"at
Graz,\\ Steyrergasse 30, 8010 Graz, Austria}
\address{Department of Mathematics, Stanford University, 450 Jane Stanford Way,\\ 
Stanford CA 94305-2125, US}
\email{behrndt@tugraz.at,jbehrndt@stanford.edu}
\urladdr{www.math.tugraz.at/$\sim$behrndt/}

\dedicatory{
Happy Birthday, Henk$\,$!\\
With great pleasure I dedicate this small note to my good friend, beer buddy, colleague,
and coauthor Henk de Snoo on the occasion of his 75th birthday.}

\begin{abstract}
The compression of the resolvent of a non-self-adjoint Schr\"{o}\-din\-ger operator $-\Delta+V$ onto a subdomain $\Omega\subset\dR^n$ 
is expressed in a Kre\u{\i}n-Na\u{\i}mark type formula, where the Dirichlet realization on $\Omega$, the Dirichlet-to-Neumann maps,
and certain solution operators of closely related boundary value problems
on $\Omega$ and $\dR^n\setminus\overline\Omega$ are being used. 
In a more abstract operator theory framework this topic is closely connected and very much inspired by the so-called coupling method that has been 
developed for the self-adjoint case by Henk de Snoo and his coauthors.
\end{abstract}


\maketitle

\section{Introduction}

Let $V\in L^\infty(\dR^n)$, $n\geq 2$, be a real or complex function and consider the Schr\"{o}dinger operator
\begin{equation}\label{aop}
 A=-\Delta+V,\qquad \dom A=H^2(\dR^n),
\end{equation}
in $L^2(\dR^n)$. Note that, in general, this operator is non-self-adjoint in $L^2(\dR^n)$.
Assume that $\Omega\subset\dR^n$ is a bounded domain with $C^2$-smooth boundary, let $P_\Omega$ be the orthogonal
projection in $L^2(\dR^n)$ onto $L^2(\Omega)$, and denote by $\iota_\Omega$ the natural embedding of $L^2(\Omega)$ into $L^2(\dR^n)$. 
The aim of this note is to derive an expression for the compression of the resolvent of $A$ onto $L^2(\Omega)$, that is,  
\begin{equation}\label{compa}
 P_\Omega (A-\lambda)^{-1}\iota_\Omega,\qquad \lambda\in\rho(A).
\end{equation}

In the special case that the potential $V$ is real the Schr\"{o}dinger operator \eqref{aop} is self-adjoint in $L^2(\dR^n)$ and hence, from a physical point of view, 
the operator $A$ can be seen as the Hamiltonian of a 
closed quantum system. In this situation the compressed resolvent \eqref{compa} can be interpreted as a family of resolvents 
of non-selfadjoint operators in $L^2(\Omega)$ modeling an open quantum system, and vice versa the operator $A$ in \eqref{aop} can be viewed as the Hamiltonian describing 
the natural closed extension of an open quantum system.
In the self-adjoint context it also follows from abstract operator theory principles that the compressed resolvent \eqref{compa} 
can be described via the Kre\u{\i}n-Na\u{\i}mark formula or can be seen as a \v{S}traus family of extensions of a symmetric operator 
in $L^2(\Omega)$, see, e.g., \cite[Chapter 2.7]{BHS20}, the contributions \cite{BS09,DHMS00,DHMS06,DHMS09,DHMS12,HKS97,HKS98} 
by Henk de Snoo and his coauthors, and also the classical works \cite{K44,K46,KL77,LT77,S65}.
However, it is of particular interest to determine the various operators and mappings that appear in 
the classical abstract Kre\u{\i}n-Na\u{\i}mark formula for the present case of a Schr\"{o}dinger operator; for real potentials $V$ 
an explicit expression for the compressed resolvent \eqref{compa} was given in \cite[Theorem 8.6.3]{BHS20} and for Lipschitz subdomains of 
Riemann manifolds in \cite[Corollary 5.5]{BDGM18}.

The main purpose of this note is to show that also in the general case of a non-self-adjoint Schr\"{o}dinger operator \eqref{aop} 
(that is, the values of the potential $V\in L^\infty(\dR^n)$ are not real a.e.) the compressed
resolvent \eqref{compa} is given by
\begin{equation}\label{knf}
 (A_\Omega-\lambda)^{-1}-\gamma_\Omega(\lambda)\bigl(M(\lambda)+\tau(\lambda)\bigr)^{-1}\widetilde\gamma_\Omega(\bar\lambda)^*,
\end{equation}
where $A_\Omega=-\Delta+V_\Omega$ is the Dirichlet realization in $L^2(\Omega)$ and $V_\Omega\in L^\infty(\Omega)$ is the restriction of $V$ onto $\Omega$.
Furthermore, $M$ and $\tau$ turn out to be the (minus) $\lambda$-dependent Dirichlet-to-Neumann maps corresponding to the differential expression $-\Delta+V$ on $\Omega$ and
on $\dR^n\setminus\overline\Omega$, respectively, and $\gamma_\Omega(\lambda)$ and $\widetilde\gamma_\Omega(\bar\lambda)$ are closely related 
solution operators, also called Poisson operators in the theory of elliptic PDEs. Our analysis is strongly inspired by the abstract coupling method 
and other boundary triple techniques, 
which were originally developed for the self-adjoint case in \cite{DHMS00,DHMS09} (see also \cite{MM97,MM99,MM02,MM03} for dual pairs) 
and some more explicit preparatory results on extension theory of non-self-adjoint Schr\"{o}dinger operators from \cite{BGHN16}. 
We do not make an attempt here to develop a systematic 
study in the non-self-adjoint context, but instead we derive \eqref{knf} in a goal-oriented way using mostly PDE-techniques such as trace maps, the second Green identity, 
and well posedness of boundary value problems.

We also mention that the analysis and spectral theory of non-self-adjoint Schr\"{o}\-dinger operators has attracted a lot of attention in the recent past. In particular,
eigenvalue bounds, Lieb-Thirring inequalities, and other spectral properties of Schr\"{o}dinger operators with complex potentials
were derived in, e.g., \cite{B17,BST17,C19,DHK13,FKV18,F11,F18,FLLS06,FS17,LS09}. The resolvent formula in Theorem~\ref{resthm} below and the compressed resolvent 
\eqref{compa}--\eqref{knf} are intimately connected with the spectral analysis of Schr\"{o}dinger operators as in \eqref{aop}.
Roughly speaking, the isolated eigenvalues of $A$ coincide with the 
isolated singularities of the function $\lambda\mapsto (M(\lambda)+\tau(\lambda))^{-1}$ and also other spectral data of $A$ can be characterized with the limit behaviour
of this function; cf. \cite{BHMNW09,BMNW08,MM02}.

\subsection*{Acknowledgement.} 
Jussi Behrndt gratefully acknowledges support for the Distinguished Visiting Austrian Chair at Stanford University by the Europe Center and the
Freeman Spogli Institute for International Studies.

\subsection*{Data Availability Statement.} 
Data sharing not applicable to this article as no datasets were generated or analysed during the current study.

\section{Preparations}

Let $\Omega\subset\dR^n$, $n\geq 2$, be a bounded domain with $C^2$-smooth boundary $\Sigma=\partial\Omega$ and denote the outward normal unit vector by $\nu$.
Furthermore, let $V_\Omega\in L^\infty(\Omega)$ be a real or complex function. We consider the
differential expressions 
\begin{equation}
\cL_\Omega=-\Delta + V_\Omega\quad\text{and}\quad \widetilde \cL_\Omega=-\Delta+\overline V_\Omega,
\end{equation}
which are formally adjoint to each other. Recall that the trace mapping $C^\infty(\overline\Omega)\ni f\mapsto\{ f\vert_{\Sigma}, \partial_\nu f_{\Sigma}\}$
can be extended to a continuous surjective mapping
\begin{equation}
 H^2(\Omega)\ni f\mapsto \bigl\{\Gamma_\Omega^D f,\Gamma_\Omega^N f \bigr\}\in H^{3/2}(\Sigma)\times H^{1/2}(\Sigma),
\end{equation}
and that for $f,g\in H^2(\Omega)$ the second Green identity in this context reads as
\begin{equation}\label{green2}
 (\cL_\Omega f,g)_{L^2(\Omega)}-(f, \widetilde\cL_\Omega g)_{L^2(\Omega)}=(\Gamma_\Omega^D f,\Gamma_\Omega^N g)_{L^2(\Sigma)}-
 (\Gamma_\Omega^N f,\Gamma_\Omega^D g)_{L^2(\Sigma)};
\end{equation}
here $H^2(\Omega)$ and $H^t(\Sigma)$, $t=\frac{1}{2},\frac{3}{2}$, denote the usual $L^2$-based Sobolev spaces on $\Omega$ and $\Sigma$, respectively.

In the following the Dirichlet operators
\begin{equation*}
\begin{split}
 A_\Omega&=-\Delta + V_\Omega,\qquad\dom A_\Omega=\bigl\{f\in H^2(\Omega):\Gamma_\Omega^D f=0\bigr\},\\
 \widetilde A_\Omega&=-\Delta + \overline V_{\!\Omega},\qquad\dom \widetilde A_\Omega=\bigl\{f\in H^2(\Omega):\Gamma_\Omega^D f=0\bigr\},
\end{split}
 \end{equation*}
will be useful. Both operators $A_\Omega$ and $\widetilde A_\Omega$ are closed, densely defined in $L^2(\Omega)$, and adjoint to each other,
\begin{equation}\label{adjoints}
 A_\Omega^*=\widetilde A_\Omega.
\end{equation}
Moreover, as $\Omega$ is a bounded domain it follows from the compactness of the embedding $H^2(\Omega)\hookrightarrow L^2(\Omega)$ that the resolvents 
of $A_\Omega$ and $\widetilde A_\Omega$ are compact operators in $L^2(\Omega)$. Note also that \eqref{adjoints} implies $\lambda\in\rho(A_\Omega)$ if and only if 
$\bar\lambda\in\rho(\widetilde A_\Omega)$.

For $\lambda\in\rho(A_\Omega)$ and $\varphi\in H^{3/2}(\Sigma)$ we consider the boundary value problem 
\begin{equation}\label{bvp1}
 (\cL_\Omega-\lambda)f=0,\qquad \Gamma_\Omega^D f=\varphi,
\end{equation}
and analogously for $\mu\in\rho(\widetilde A_\Omega)$ and $\psi\in H^{3/2}(\Sigma)$ we consider the boundary value problem 
\begin{equation}\label{bvp2}
 (\widetilde\cL_\Omega-\mu)g=0,\qquad \Gamma_\Omega^D g=\psi.
\end{equation}
From the the assumptions $\lambda\in\rho(A_\Omega)$ and 
$\mu\in\rho(\widetilde A_\Omega)$ and the fact that $\Gamma_\Omega^D:H^2(\Omega)\rightarrow H^{3/2}(\Sigma)$ is onto  it follows that both boundary value problems \eqref{bvp1} 
and \eqref{bvp2} admit a unique solution 
$f(\varphi,\lambda)\in H^2(\Omega)$ and $g(\psi,\mu)\in H^2(\Omega)$. We shall use the notation 
\begin{equation}
 \gamma_\Omega(\lambda)\varphi:= f(\varphi,\lambda) \quad\text{and}\quad \widetilde\gamma_\Omega(\mu)\psi:= g(\psi,\mu)
\end{equation}
for the solution operators of \eqref{bvp1} and \eqref{bvp2}. The next lemma is essentially a consequence of
the second Green identity \eqref{green2}; we leave the proof to the reader.

\begin{lemma}\label{lemmagam}
 For $\lambda\in\rho(A_\Omega)$ and $\mu\in\rho(\widetilde A_\Omega)$ the solution operators $\gamma(\lambda)$ and $\widetilde\gamma(\mu)$
 are bounded from $L^2(\Sigma)$ to $L^2(\Omega)$ with dense domain $H^{3/2}(\Sigma)$ and range contained in $H^2(\Omega)$.
 Their adjoints are everywhere defined bounded operators from $L^2(\Omega)$ to $L^2(\Sigma)$ 
 given by
 \begin{equation*}
  \gamma_\Omega(\lambda)^*=-\Gamma_\Omega^N (\widetilde A_\Omega-\bar\lambda)^{-1} \quad\text{and}\quad \widetilde\gamma_\Omega(\mu)^*=-\Gamma_\Omega^N (A_\Omega-\bar\mu)^{-1}.
 \end{equation*}
 In particular, one has $\ran \gamma_\Omega(\lambda)^*= H^{1/2}(\Sigma)= \ran\widetilde\gamma_\Omega(\mu)^*$.
\end{lemma}

A further important object in our study will be the (minus) Dirichlet-to-Neumann map $M(\cdot)$ corresponding to $\cL_\Omega$. Let again $\lambda\in\rho(A_\Omega)$ and 
$\varphi\in H^{3/2}(\Sigma)$. Then the operator $M(\lambda)$ is defined by
\begin{equation*}
 M(\lambda)\varphi=-\Gamma_\Omega^N \gamma_\Omega(\lambda)\varphi=-\Gamma_\Omega^N f(\varphi,\lambda),
\end{equation*}
where $\gamma_\Omega(\lambda)\varphi= f(\varphi,\lambda)\in H^2(\Omega)$ is the unique solution of the boundary value problem \eqref{bvp1}. The Dirichlet-to-Neumann 
map is an operator mapping $H^{3/2}(\Sigma)$ into $H^{1/2}(\Sigma)$, but can also be viewed as a densely defined unbounded (nonclosed) operator
in $L^2(\Sigma)$.

Besides the domain $\Omega$ and the operators introduced above we shall also make use of their counterparts acting on the unbounded (exterior) domain 
$\Omega':=\dR^n\setminus\overline\Omega$ with $C^2$-smooth boundary $\Sigma=\partial\Omega'$. 
The operators will be denoted in the same way, except that 
we shall use the subindex $\Omega'$ instead of $\Omega$, e.g., $A_{\Omega'}$ stands for the Dirichlet realization of $\cL_{\Omega'}=-\Delta+V_{\Omega'}$ in $L^2(\Omega')$.
For the (minus) Dirichlet-to-Neumann map 
we will use the symbol $\tau(\cdot)$; note that the values $\tau(\lambda)$ are well defined for all $\lambda\in\rho(A_{\Omega'})$. The above statements all remain valid on the unbounded domain, with the only exception that 
the embedding $H^2(\Omega')\hookrightarrow L^2(\Omega')$ is not compact and the resolvents 
of $A_{\Omega'}$ and $\widetilde A_{\Omega'}$ are not compact in $L^2(\Omega')$.

In our considerations we shall sometimes make use of a vector notation $(f,f')^\top:\dR^n\rightarrow\dC$
for functions $f:\Omega\rightarrow\dC$ and $f':\Omega\rightarrow\dC$.
The following simple observation will be useful in the proof of Theorem~\ref{resthm} in the next section. 

\begin{lemma}\label{h2lemma}
Let $f\in H^2(\Omega)$ and $f'\in H^2(\Omega')$. Then  
\begin{equation}
\begin{pmatrix}
 f\\ f'
\end{pmatrix}
\in H^2(\dR^n)\quad\text{if and only if}\quad \Gamma_\Omega^D f=\Gamma_{\Omega'}^D f'\,\,\text{and}\,\, \Gamma_\Omega^N f=-\Gamma_{\Omega'}^N f'.
\end{equation}
\end{lemma}

\begin{proof}
 The implication $(\Rightarrow)$ is clear from the definition of the trace maps and the implication $(\Leftarrow)$ can be viewed as a consequence of the self-adjointness
 of the Laplacian $H=-\Delta$ defined on $\dom H=H^2(\dR^n)$ and the second Green identity. In fact,
 let $\widetilde g\in \dom H$ and assume that $f\in H^2(\Omega)$ and $f'\in H^2(\Omega')$ satisfy 
 $\Gamma_\Omega^D f=\Gamma_{\Omega'}^D f'$ and $\Gamma_\Omega^N f=-\Gamma_{\Omega'}^N f'$. 
 For $\widetilde g=(g,g')^\top\in H^2(\dR^n)$ we also have $\Gamma_\Omega^D g=\Gamma_{\Omega'}^D g'$ and $\Gamma_\Omega^N g=-\Gamma_{\Omega'}^N g'$, and hence
 for $\widetilde f=(f,f')^\top\in H^2(\Omega)\times H^2(\Omega')$ it follows from \eqref{green2} that
 \begin{equation*}
  \begin{split}
   &(H \widetilde g,\widetilde f)_{L^2(\dR^n)}-(\widetilde g,-\Delta\widetilde f)_{L^2(\dR^n)}\\
   &\quad=(-\Delta g,f)_{L^2(\Omega)}-( g,-\Delta f)_{L^2(\Omega)}+(-\Delta g',f')_{L^2(\Omega')}-( g',-\Delta f')_{L^2(\Omega')}\\
   &\quad=(\Gamma_\Omega^D g,\Gamma_\Omega^N f + \Gamma_{\Omega'}^N f')_{L^2(\Sigma)}-(\Gamma_\Omega^N g,\Gamma_\Omega^D f-\Gamma_{\Omega'}^D f')_{L^2(\Sigma)}\\
   &\quad=0.
  \end{split}
 \end{equation*}
Therefore, $\widetilde f=(f,f')^\top\in\dom H^*=\dom H=H^2(\dR^n)$.
\end{proof}

\section{A formula for the resolvent of the operator $A$}

In this section we obtain a Kre\u{\i}n type formula for the resolvent of the non-self-adjoint Schr\"{o}dinger operator in \eqref{aop},
and as an immediate consequence we conclude the form \eqref{compa} of the compressed resolvent.
The construction is based on the abstract coupling method developed in \cite{DHMS00,DHMS09}, but is made more explicit here in the context
of differential operators.

In the following let $A$ be the non-self-adjoint Schr\"{o}dinger operator in \eqref{aop}, let $A_{\Omega}$ and $A_{\Omega'}$ be the Dirichlet realizations of 
$\cL=-\Delta+V$ in $L^2(\Omega)$ and $L^2(\Omega')$, respectively, and denote by $M(\cdot)$ and $\tau(\cdot)$ the (minus) Dirichlet-to-Neumann maps on $\Omega$ and $\Omega'$.
The Dirichlet realizations of  $\widetilde\cL=-\Delta+\overline V$ in $L^2(\Omega)$ and $L^2(\Omega')$ are denoted by $\widetilde A_\Omega$ and $\widetilde A_{\Omega'}$, respectively.
The next lemma is needed in the proof of our resolvent formula in Theorem~\ref{resthm} below. In the self-adjoint context this lemma was shown in \cite[Lemma 8.6.1]{BHS20}.
As the proof remains the same in the general non-self-adjoint situation we do not repeat it here.

\begin{lemma}\label{mtlem}
 For $\lambda\in\rho(A_\Omega)\cap\rho(A_{\Omega'})\cap\rho(A)$ the operator 
 \begin{equation}\label{mt}
  M(\lambda)+\tau(\lambda): H^{3/2}(\Sigma)\rightarrow H^{1/2}(\Sigma)
 \end{equation}
 is bijective.
\end{lemma}

For our purposes it is convenient to use the notation $A_{\Omega,\Omega'}:=A_\Omega \times A_{\Omega'}$ and we regard  $A_{\Omega,\Omega'}$ as a closed operator
in $L^2(\dR^n)=L^2(\Omega)\times L^2(\Omega')$. 
Note that $\rho(A_{\Omega,\Omega'})=\rho(A_\Omega)\cap\rho(A_{\Omega'})$  and
$$
(A_{\Omega,\Omega'}-\lambda)^{-1}=\begin{pmatrix} (A_\Omega-\lambda)^{-1} & 0 \\ 0 & (A_{\Omega'}-\lambda)^{-1} \end{pmatrix},\quad \lambda\in\rho(A_\Omega)\cap\rho(A_{\Omega'}).
$$
Furthermore, we set
\begin{equation}
\begin{split}
 \gamma_{\Omega,\Omega'}(\lambda)&=\begin{pmatrix}\gamma_\Omega(\lambda) & 0 \\ 0 & \gamma_{\Omega'}(\lambda)\end{pmatrix},\quad \lambda\in\rho(A_\Omega)\cap\rho(A_{\Omega'}),\\
 \widetilde\gamma_{\Omega,\Omega'}(\lambda)&=\begin{pmatrix}\widetilde\gamma_\Omega(\lambda) & 0 \\ 0 & \widetilde\gamma_{\Omega'}(\lambda)\end{pmatrix},\quad 
 \lambda\in\rho(\widetilde A_\Omega)\cap\rho(\widetilde A_{\Omega'}),
 \end{split}
 \end{equation}
and for $\lambda\in\rho(A_\Omega)\cap\rho(A_{\Omega'})\cap\rho(A)$ we define
\begin{equation}\label{theta}
 \Theta(\lambda):=\begin{pmatrix} (M(\lambda)+\tau(\lambda))^{-1} & (M(\lambda)+\tau(\lambda))^{-1} \\ (M(\lambda)+\tau(\lambda))^{-1} & (M(\lambda)+\tau(\lambda))^{-1}
 \end{pmatrix}.
\end{equation}

%
%

The next theorem is the main result of this note. We express the resolvent of the Schr\"{o}dinger operator $A$ in terms of the resolvent of the orthogonal
sum $A_{\Omega,\Omega'}$ of the Dirichlet realizations and a perturbation term, which contains the Dirichlet-to-Neumann maps $M(\cdot)$ and $\tau(\cdot)$,
the solution operators $\gamma_\Omega(\cdot)$ and $\gamma_{\Omega'}(\cdot)$, and their adjoints. In particular, since the solutions operators are 
analytic on the resolvent sets $\rho(A_\Omega)$ and $\rho(A_{\Omega'})$, respectively, it follows that the poles of the resolvent of $A$ (and hence also the isolated 
eigenvalues) in $\rho(A_\Omega)\cap\rho(A_{\Omega'})$ coincide with the isolated singularities of the function $\Theta(\cdot)$ in \eqref{theta}.

\begin{theorem}\label{resthm}
 For $\lambda\in\rho(A_\Omega)\cap\rho(A_{\Omega'})\cap\rho(A)$ the resolvent formula
 \begin{equation}\label{jaja}
  (A-\lambda)^{-1}=(A_{\Omega,\Omega'}-\lambda)^{-1}-\gamma_{\Omega,\Omega'}(\lambda)\Theta(\lambda)\widetilde\gamma_{\Omega,\Omega'}(\bar\lambda)^*
 \end{equation}
is valid. In particular, the compression of the resolvent of $A$ onto $L^2(\Omega)$ is given by 
\begin{equation}\label{compa2}
P_\Omega(A-\lambda)^{-1}\iota_\Omega= (A_\Omega-\lambda)^{-1}-\gamma_\Omega(\lambda)\bigl(M(\lambda)+\tau(\lambda)\bigr)^{-1}\widetilde\gamma_\Omega(\bar\lambda)^*.
\end{equation}
\end{theorem}

\begin{proof}
Let $f\in L^2(\Omega)$ and $f'\in L^2(\Omega')$, and consider  
\begin{equation}\label{gg}
 \begin{pmatrix} g \\ g'\end{pmatrix} =(A_{\Omega,\Omega'}-\lambda)^{-1}\begin{pmatrix} f \\ f'\end{pmatrix} 
 -\gamma_{\Omega,\Omega'}(\lambda)\Theta(\lambda)\widetilde\gamma_{\Omega,\Omega'}(\bar\lambda)^*\begin{pmatrix} f \\ f'\end{pmatrix}, 
\end{equation}
that is,
\begin{equation}\label{ggg}
 \begin{split}
  g&=(A_\Omega-\lambda)^{-1}f- \gamma_\Omega(\lambda)\bigl(M(\lambda)+\tau(\lambda)\bigr)^{-1}
     \bigl(\widetilde \gamma_\Omega(\bar\lambda)^* f + \widetilde\gamma_{\Omega'}(\bar\lambda)^*f'\bigr),\\
     g'&=(A_{\Omega'}-\lambda)^{-1}f'- \gamma_{\Omega'}(\lambda)\bigl(M(\lambda)+\tau(\lambda)\bigr)^{-1}
     \bigl(\widetilde \gamma_\Omega(\bar\lambda)^* f + \widetilde\gamma_{\Omega'}(\bar\lambda)^* f'\bigr).
  \end{split}
\end{equation}
It follows from $\ran\widetilde\gamma_\Omega(\bar\lambda)^*=H^{1/2}(\Sigma)=\ran\widetilde\gamma_{\Omega'}(\bar\lambda)^*$ and Lemma~\ref{mtlem} that
the products on the right-hand side of \eqref{gg}--\eqref{ggg} are well-defined.
Since $\dom A_\Omega\subset H^2(\Omega)$ and $\ran\gamma_\Omega(\lambda)\subset H^2(\Omega)$ by Lemma~\ref{lemmagam} one has $g\in H^2(\Omega)$. In the same way it follows that 
$g'\in H^2(\Omega')$. Moreover, as $A_\Omega$ and $A_{\Omega'}$ are Dirichlet realizations we have 
$$\Gamma_\Omega^D(A_\Omega-\lambda)^{-1}f=0\quad\text{and}\quad \Gamma_{\Omega'}^D(A_{\Omega'}-\lambda)^{-1}f'=0.$$ Together with the definition of the solution 
operators $\gamma_\Omega(\lambda)$ and $\gamma_{\Omega'}(\lambda)$ this leads to
\begin{equation}
  \Gamma_\Omega^D g=-\bigl(M(\lambda)+\tau(\lambda)\bigr)^{-1}\bigl(\widetilde \gamma_\Omega(\bar\lambda)^* f + \widetilde\gamma_{\Omega'}(\bar\lambda)^* f'\bigr)
\end{equation}
and
\begin{equation}
  \Gamma_{\Omega'}^D g'=-\bigl(M(\lambda)+\tau(\lambda)\bigr)^{-1}\bigl(\widetilde \gamma_\Omega(\bar\lambda)^* f + \widetilde\gamma_{\Omega'}(\bar\lambda)^* f'\bigr).
\end{equation}
Using Lemma~\ref{lemmagam} and the definition of the Dirichlet-to-Neumann maps $M(\cdot)$ and $\tau(\cdot)$ we find
\begin{equation}
 \begin{split}
  \Gamma_\Omega^N g&=\Gamma_\Omega^N (A_\Omega-\lambda)^{-1}f \\
                   &\qquad -\Gamma_\Omega^N \gamma_\Omega(\lambda) \bigl(M(\lambda)+\tau(\lambda)\bigr)^{-1}
                     \bigl(\widetilde \gamma_\Omega(\bar\lambda)^* f + \widetilde\gamma_{\Omega'}(\bar\lambda)^* f'\bigr)\\
                   &=-\widetilde \gamma_\Omega(\bar\lambda)^* f + M(\lambda) \bigl(M(\lambda)+\tau(\lambda)\bigr)^{-1}
                     \bigl(\widetilde \gamma_\Omega(\bar\lambda)^* f + \widetilde\gamma_{\Omega'}(\bar\lambda)^* f'\bigr)
 \end{split}
\end{equation}
and 
\begin{equation}
 \begin{split}                     
  \Gamma_{\Omega'}^N g'&=\Gamma_{\Omega'}^N (A_{\Omega'}-\lambda)^{-1}f' \\
                       &\qquad -\Gamma_{\Omega'}^N \gamma_{\Omega'}(\lambda) \bigl(M(\lambda)+\tau(\lambda)\bigr)^{-1}
                     \bigl(\widetilde \gamma_\Omega(\bar\lambda)^* f + \widetilde\gamma_{\Omega'}(\bar\lambda)^* f'\bigr)\\
                   &=-\widetilde \gamma_{\Omega'}(\bar\lambda)^* f' + \tau(\lambda) \bigl(M(\lambda)+\tau(\lambda)\bigr)^{-1}
                     \bigl(\widetilde \gamma_\Omega(\bar\lambda)^* f + \widetilde\gamma_{\Omega'}(\bar\lambda)^* f'\bigr).
 \end{split}
\end{equation}
Therefore, we have
\begin{equation*}
 \Gamma_\Omega^D g=\Gamma_{\Omega'}^D g'\quad\text{and}\quad \Gamma_\Omega^N g+\Gamma_{\Omega'}^N g'=0,
\end{equation*}
and now Lemma~\ref{h2lemma} implies that the function in \eqref{gg} is in $H^2(\dR^n)=\dom A$. 
As $(\cL_\Omega-\lambda)\gamma_\Omega(\lambda)\varphi=0$ and $(\cL_{\Omega'}-\lambda)\gamma_{\Omega'}(\lambda)\psi=0$ for all $\varphi,\psi\in H^{3/2}(\Sigma)$ 
it is also clear that
\begin{equation*}
\begin{split}
 (A-\lambda)\begin{pmatrix} g \\ g'\end{pmatrix}&=(-\Delta +V-\lambda)\begin{pmatrix} g \\ g'\end{pmatrix}\\
 &=\begin{pmatrix} \cL_\Omega-\lambda \\ \cL_{\Omega'}-\lambda \end{pmatrix}
 \begin{pmatrix} (A_\Omega-\lambda)^{-1}f \\ (A_{\Omega'}-\lambda)^{-1}f'\end{pmatrix} \\
 &\qquad
 -\begin{pmatrix} \cL_\Omega-\lambda \\ \cL_{\Omega'}-\lambda \end{pmatrix}
 \begin{pmatrix}\gamma_\Omega(\lambda) & 0 \\ 0 & \gamma_{\Omega'}(\lambda)\end{pmatrix}
 \Theta(\lambda)\widetilde\gamma_{\Omega,\Omega'}(\bar\lambda)^*\begin{pmatrix} f \\ f'\end{pmatrix}
 \\
 &=\begin{pmatrix} f \\ f'\end{pmatrix},
 \end{split}
 \end{equation*}
which leads to \eqref{jaja}. The formula \eqref{compa2} for the compressed resolvent is an immediate consequence of \eqref{jaja}.
\end{proof}

From Lemma~\ref{lemmagam} it is clear that $\ran\widetilde\gamma_{\Omega,\Omega'}(\bar\lambda)^* = H^{1/2}(\Sigma)\times H^{1/2}(\Sigma)$, 
$\lambda\in\rho(A_\Omega)\cap\rho(A_{\Omega'})$. Furthermore, since
the embedding $H^{1/2}(\Sigma)\hookrightarrow L^2(\Sigma)$
is compact one concludes that $\widetilde\gamma_{\Omega,\Omega'}(\bar\lambda)^*$ is a compact operator from $L^2(\dR^n)$ to $L^2(\Sigma)\times L^2(\Sigma)$.
Since $(M(\lambda)+\tau(\lambda))^{-1}$, $\lambda\in\rho(A_\Omega)\cap\rho(A_{\Omega'})\cap\rho(A)$, can be extended to a bounded operator on $L^2(\Sigma)$ it follows that $\Theta(\lambda)$ in \eqref{theta} admits a 
bounded extension to $L^2(\Sigma)\times L^2(\Sigma)$. Thus, the perturbation term in the resolvent formula in Theorem~\ref{resthm} is compact, and hence 
the resolvent difference 
\begin{equation*}
(A-\lambda)^{-1}-(A_{\Omega,\Omega'}-\lambda)^{-1},\quad \lambda\in\rho(A_\Omega)\cap\rho(A_{\Omega'})\cap\rho(A),
\end{equation*}
is compact in $L^2(\dR^n)$ (and, in fact, it can be shown that the resolvent difference belongs to some Schatten--von Neumann ideal). Therefore,
well known perturbation results imply that the essential spectra of $A$ and $A_{\Omega,\Omega'}$ coincide, and as the resolvent of $A_\Omega$ is compact
we conclude
\begin{equation*}
 \sigma_{\rm ess}(A)= \sigma_{\rm ess}(A_{\Omega,\Omega'})= \sigma_{\rm ess}(A_{\Omega})\cup \sigma_{\rm ess}(A_{\Omega'})=\sigma_{\rm ess}(A_{\Omega'});
\end{equation*}
here the essential spectrum of a non-self-adjoint operator is defined as the complement of the isolated eigenvalues 
with finite algebraic multiplicities in the spectrum.

\end{document}